\documentclass[12pt]{article}

\usepackage{amsfonts}
\usepackage{amsmath}
\usepackage{amssymb}
\usepackage{amsthm}
\usepackage{amscd}

\theoremstyle{plain} 
\newtheorem{theo}{Theorem}[section]
\newtheorem{prop}[theo]{Proposition}
\newtheorem{lem}[theo]{Lemma}

\theoremstyle{definition}

\newtheorem{rk}[theo]{Remark}

\newtheorem{exa}[theo]{Example}

\numberwithin{equation}{section}

\newcommand{\from}{\colon}
\newcommand{\N}{{\mathbb N}}
\newcommand{\Z}{{\mathbb Z}}
\newcommand{\R}{{\mathbb R}}

\newcommand{\calE}{{\cal{E}}}

\newcommand{\Om}{\Omega}
\newcommand{\cH}{\check{H}}

\newcommand{\Kn}{K^{n}}
\newcommand{\calEn}{{\cal{E}}^{n}}
\newcommand{\ccalE}{\check{\cal{E}}}

\newcommand{\calD}{{\cal{D}}}
\newcommand{\calH}{{\cal{H}}}
\newcommand{\calHaux}{{\cal{H}}_{\rm aux}}
\newcommand{\Har}{{\cal{H}}_{\rm har}^J}
\newcommand{\Jn}{J^{n}_1}
\newcommand{\form}{a}
\DeclareMathOperator{\ran}{ran}
\DeclareMathOperator{\dom}{dom}
\newcommand{\reg}{\mathrm{reg}}
\DeclareMathOperator{\supp}{supp}
\newcommand{\aux}{\mathrm{aux}}
\newcommand{\e}{\mathrm{e}}
\newcommand{\Dir}{\mathrm{D}}
\newcommand{\har}{\mathrm{har}}
\newcommand{\E}{\mathbb{E}}

\begin{document}
\title{On the construction and convergence of traces of forms}
\author{\normalsize
Hichem BelHadjAli\footnote{Department of Mathematics, I.P.E.I.N. Uni. Carthage, Tunisia. E-mail: hichem.belhadjali@ipein.rnu.tn},
Ali BenAmor\footnote{Institute of transport and logistics. Uni. Sousse, Tunisia. E-mail: ali.benamor@ipeit.rnu.tn},
Christian Seifert\footnote{TU Hamburg, Institut f\"ur Mathematik, Am Schwarzenberg-Campus 3 E, 21073 Hamburg, Germany. E-mail: christian.seifert@tuhh.de},
Amina Thabet\footnote{Department of Mathematics, Faculty of Sciences of Gab\`es. Uni. Gab\`es, Tunisia. E-Mail: maino.88@live.com}
}

\date{\today}

\maketitle

\begin{quotation}
\begin{center}
\textit{Dedicated to the memory of Johannes F.\ Brasche.}
\end{center}
\end{quotation}

\begin{abstract}
  We elaborate a new method for constructing traces of quadratic forms in the framework of Hilbert and Dirichlet spaces.
  Our method relies on monotone convergence of quadratic forms and the canonical decomposition into regular and singular part.
  We give various situations where the trace can be described more explicitly and compute it for some illustrating examples.
  We then show that Mosco convergence of Dirichlet forms implies Mosco convergence of a subsequence of their approximating traces.
\end{abstract}

\textbf{MSC 2010:} 47A07, 46C05, 46C07, 47B25, 46E30.

\textbf{Keywords:} trace of forms, Dirichlet forms, Mosco convergence

\section{Introduction}

In this paper we study the construction of traces of closed positive quadratic forms $\calE$ in Hilbert spaces with respect to some given linear operator $J$.
By this we mean, starting with a closed positive quadratic form with domain in some Hilbert space $\calH$ and a linear operator $J$ with domain in the same space $\calH$ but having values in some other auxiliary Hilbert space $\calHaux$, we shall construct a new closed quadratic form in $\calHaux$. Let us stress that the mentioned problem is not new and there are various methods for constructing such a form in the literature, see \cite{Arendt1,ArendtterElstKennedySauter2014,vogt,Fukushima}. The most general construction can be found in  \cite{Arendt1}, where the authors construct an operator in $\calHaux$ starting from $\calE$ and $J$ and then of course the form. The novelty in our method consists in following the converse strategy as follows: starting with a form in $\calH$ we construct the so-called trace form in $\calHaux$ and its associated operator simultaneously, by means of approximating forms. However, we will show in Theorem \ref{thm:constructions_agree} that both constructions lead in fact to the same object. Besides we shall also focus on explicit computation of the obtained form.

Let us explain our method. Instead of using Kato--Lions method for forms we make use of monotone convergence of quadratic forms together with their canonical decomposition into a regular part and a singular one, see \cite{Simon}. This method of construction seems not to exist in the literature.
The main input at this stage is a Dirichlet principle consisting in describing the approximating forms in a variational way.
Thanks to this method we are able to compute explicitly traces of forms in many general circumstances.

In the special case of Dirichlet spaces we show, with a short and analytic proof, that traces of regular Dirichlet forms are regular Dirichlet forms as well.
We also show that Mosco-convergence of Dirichlet forms yields Mosco-convergence of a subsequence of approximating trace forms.
We refer to \cite{Mosco1994} for the corresponding notion (which will be recalled in Section \ref{sec:convergence_Dirichlet_forms} below as well).
At this stage we shall make use of the theory of convergence of sequences of Hilbert spaces and its corollaries elaborated in \cite{Kuwae}.

The concept of traces of forms goes back to Fukushima-Oshima-Takeda \cite[Section 6.2]{FOT1994,Fukushima}, where the authors initiate the construction, investigate the trace form and relate it to part of processes.
However, many proofs, especially in the non-transient case, are based on arguments making use of the theory of stochastic processes. We aim for analytic arguments.
Recently the subject gained much more interest due to a generalization of the form method by Arendt and ter Elst \cite{Arendt1}.
Since then there has been various studies of properties of traces of sectorial forms in Hilbert spaces.
In \cite{ArendtterElstKennedySauter2014} the authors rely their construction on a hidden compactness condition yielding ellipticity for the form.
In \cite{BBBMath} the construction of the trace of $\calE_1$, the form $\calE$ shifted by $1$, is given.
We will make use of the traces of $\calE_\lambda$ for all $\lambda>0$ given in this way and then take the appropriate limit for $\lambda\to 0$.
Ter Elst, Sauter and Vogt in \cite{vogt} proved a generation theorem for accretive forms under the assumption that $J$ is bounded with dense range, which extends the results of \cite{Arendt1}.
In \cite{Post2016}, Post used so-called boundary pairs (referring to the case that $J$ has a dense kernel) to construct a family of operators related to the associated operator to the trace form.
Moreover, there are applications in the context of Dirichlet forms and singular diffusions, see \cite{SeifertVoigt2011,FreibergSeifert2015}.

Traces of quadratic forms have a wide range of applications in a variety of fields.
Let us cite, among others, their connection to parts of stochastic processes established in \cite{Fukushima1980}, their relationship to the construction of Dirichlet-to-Neumann operators \cite{Arendt2,Daners} and of fractional powers of the Laplacian \cite{MolchanovOstrocskii1969,Caffarelli}.
Traces of forms also appear in the study of problems related to large coupling convergence and spectral asymptotics \cite{BBBMath, BBBT}.

The paper is organized as follows. In Section \ref{sec:Traces_quadratic_forms} we introduce the setup for quadratic forms in Hilbert spaces, prove a Dirichlet principle for the approximating forms and construct the trace via monotone convergence and regular parts.
We then focus on special situations, where we can compute the trace more explicitly.
In Section \ref{sec:examples} we apply our method to various examples and calculate the corresponding traces.
This includes the square root of the Laplacian as obtained in \cite{Caffarelli} revisited in the context of forms, but also traces on (maybe small) subsets, wich can correspond to singular diffusions; cf.\ \cite{SeifertVoigt2011, FreibergSeifert2015}.
Starting from Section \ref{sec:Traces_Dirichlet_forms} we focus on Dirichlet forms.
First, we show that the trace of a regular Dirichlet form is a regular Dirichlet form again (when interpreted in the right space).
We also relate our method of construction with the probabilistic one in \cite[Section 6.2]{Fukushima}, and show that these two traces coincide.
The final Section \ref{sec:convergence_Dirichlet_forms} is devoted to properties of sequences of Dirichlet forms.
Here we prove that Mosco convergence implies Mosco convergence of a subsequence of approximating trace forms.

\section{Traces of quadratic forms in Hilbert spaces}
\label{sec:Traces_quadratic_forms}

Let $\calH, \calHaux$ be two Hilbert spaces.
Let $(\cdot,\cdot)$ and $(\cdot,\cdot)_{\aux}$ denote the scalar products on $\calH$ and $\calHaux$, respectively.
Let $\calE$ be a closed positive quadratic form with domain $\calD\subseteq \calH$.
For $u\in\calD$ we abbreviate $\calE[u]:=\calE(u,u)$ and for every $\lambda>0$ set
\[
  \calE_\lambda[u] := \calE[u] + \lambda\|u\|^2.
\]
Assume we are given a linear operator $J\from \dom J\subseteq\calD\to \calHaux$ with dense range
such that $J$ is closed in $(\calD,\calE_1^{1/2})$.
For $\lambda>0$ we define $J_\lambda\from \dom J\subseteq (\calD,\calE_\lambda^{1/2})\to \calHaux$ by $J_\lambda u :=Ju$.
Let $(\ker J_\lambda)^{\perp_{\calE_\lambda}}$ be the $\calE_\lambda$-orthogonal complement of $\ker J_\lambda$ and let $P_\lambda$ the $\calE_\lambda$-orthogonal projection onto $(\ker J_\lambda)^{\perp_{\calE_{\lambda}}}$.

For $\lambda>0$ we construct a new family of closed positive densely defined quadratic forms as follows (see \cite[Theorem 1.1]{BBBMath})
\begin{equation}
  \dom \check{\calE}_\lambda := \ran J,\quad \check{\calE}_\lambda[Ju]:= {\calE_\lambda}[P_\lambda u] \quad\text{for all } u\in \dom J.
  \label{construction1}
\end{equation}

Let $\check{H}_\lambda$ be the positive self-adjoint operator associated with $\check{\calE}_\lambda$. We emphasize that, if moreover $J$ is densely defined then from \cite[Theorem 1.1]{BBBMath} once again we obtain
\begin{equation}
  \check{H}_\lambda = (J_\lambda J_\lambda^*)^{-1}.
  \label{checkH}
\end{equation}

We start with a result that is of major importance for our construction of traces of quadratic forms and which expresses the variational aspect of the forms $\check{\calE}_\lambda$.

\begin{theo}[Dirichlet principle]
\label{thm:Dirichlet_principle}
  Let $\lambda>0$, $u\in \dom J$. Then
  \[
  \check{\calE}_\lambda[Ju] = \inf\{{\calE_\lambda}[v]:\; v\in \dom J,\ Jv=Ju\}.
  \]
  Moreover, $\check{\calE}_\lambda \leq \check{\calE}_\mu$ for $\lambda\leq \mu$.
\end{theo}

In the proof of the Dirichlet principle we will make use of the following lemma.

\begin{lem}
\label{lem:projection}
  Let $(\calD,q)$ be a Hilbert space, $J\from \dom J\subseteq \calD \to \calHaux$ linear and closed. Let $P$ be the $q$-orthogonal projection onto $(\ker J)^{\bot_q}$. Let $u\in \dom J$. Then $Pu\in \dom J$ and $Ju = JPu$.
\end{lem}

\begin{proof}
  Note that closed linear operators have closed kernels. Hence, $\ker J$ is closed.
  Since $P$ is an $q$-orthogonal projection, we obtain
  \[u-Pu \in (\ran P)^{\bot_q} = \bigl((\ker J)^{\bot_q}\bigr)^{\bot_q} = \overline{\ker J} = \ker J\subseteq \dom J.\]
  Since $u\in\dom J$, we obtain $Pu\in\dom J$ and $Ju - JPu = J(u-Pu) = 0$.
\end{proof}

\begin{proof}[Proof of Theorem \ref{thm:Dirichlet_principle}]
  By Lemma \ref{lem:projection} we have $P_\lambda u\in \dom J$, and $Ju=JP_\lambda u$.
  Thus,
  \[
    \inf\{\calE_\lambda[v]:\; v\in \dom J,\ Jv=Ju\}\leq {\calE_\lambda}[P_\lambda u]=\check{\calE}_\lambda[Ju].
  \]
  On the other hand, owing to the fact that $P_\lambda$ is an orthogonal projection w.r.t.\ $\calE_\lambda$ we get
  \[
    \calE_\lambda[v]\geq \calE_\lambda[P_\lambda v] \quad \text{for all } v\in \dom J.
  \]
  Now if $v\in \dom J$ and $Jv=Ju$ then we obtain $P_\lambda v = P_\lambda u$ and therefore
  \[
  \inf\{\calE_\lambda[v]:\; v\in \dom J,\ Jv=Ju\}\geq \calE_\lambda[P_\lambda u]= \check{\calE}_\lambda[Ju].
  \]
  Since $(\calE_\lambda)_{\lambda>0}$ is a monotone increasing family, also $(\check{\calE}_\lambda)_\lambda$ is monotone increasing.
\end{proof}

\begin{rk}
  Let $Q$ be a densely defined positive quadratic form on a Hilbert space $\calH$.
  Then $Q$ can be uniquely decomposed into $Q=Q_\reg + Q_{\rm sing}$, such that
  $Q_\reg$ is the largest positive densely defined closable quadratic form dominated by $Q$.
  In particular, if $Q$ is closable then $Q_\reg=Q$. The form $Q_\reg$ is called the \emph{regular part} of $Q$.
  See \cite{Simon,Mosco1994} for more details on this decomposition.
\end{rk}

\begin{theo}
  There exists a positive self-adjoint operator $\check{H}$ in $\calHaux$ such that
  \[
    \lim_{\lambda\downarrow 0}(\check {H}_\lambda +1)^{-1}=(\check H + 1)^{-1} \;\text{strongly}.
  \]
  Furthermore, defining $\check{\calE}_0$ in $\calHaux$ by
  \[
    \dom \check{\calE}_0 := \ran J,\quad \check{\calE}_0[Ju]:= \lim_{\lambda\downarrow 0}\check{\calE}_\lambda[Ju] \quad\text{for all } u\in \dom J,
  \]
  then $\check H$ is the self-adjoint operator associated with the closure of $(\check{\calE}_0)_\reg$.
  In particular, if $\check{\calE}_0$ is closable then $\check H$ is the self-adjoint operator associated with the closure of $\check{\calE}_0$.
  \label{construction}
\end{theo}

\begin{proof}
  For $\lambda>0$ the form $\check{\calE}_\lambda$ is densely defined, positive and closed.
  By Theorem \ref{thm:Dirichlet_principle} the family $(\check{\calE}_\lambda)_\lambda$ is monotone increasing.
  Making use  of \cite[Theorem VIII.3.11]{Kato} we conclude that there is a positive self-adjoint operator in $\calHaux$, which we denote by  $\check H$, such that
  \[
    \lim_{\lambda\downarrow 0}(\check {H}_\lambda +1)^{-1}=(\check H + 1)^{-1} \;\text{strongly}.
  \]
  Moreover from \cite[Theorem 3.2]{Simon} we infer that $\check H$ is the self-adjoint operator associated with the closure of $(\check{\calE}_0)_\reg$.
  The last claim of the theorem follows from the definition of the regular part of a quadratic form.
\end{proof}

From now on we let $\check\calE$ be the densely defined positive closed quadratic form associated to $\check H$ via the second representation theorem \cite[Theorem VI.2.23]{Kato}:
\[
  \check\calD:=\dom \check\calE :=\dom \check{H}^{1/2},\quad \check\calE[\psi]:=(\check{H}^{1/2}\psi,\check{H}^{1/2}\psi)_{\aux}.
\]
We shall call $\check\calE$ the {\em trace} of $\calE$ with respect to $J$. Note that $\check{\calE} = \overline{(\check{\calE}_0)_\reg}$.
Let us quote that from the definition of the regular part we have $\ran J\subseteq \dom \check\calE$.
Hence the domain of $\check\calE$ is the closure of $\ran J$ w.r.t.\ $\sqrt{\check\calE[\cdot] +\|\cdot \|^2_{\aux}}$.

\begin{rk}
\label{rem:construction}
  {\rm(a)}
  Let $\lambda>0$.
  One may ask whether the trace of $\calE_\lambda$ agrees with $\check{\calE}_\lambda$ from \eqref{construction1}.
  In Proposition \ref{consistency} we will show that the construction is consistent.

  {\rm(b)}
  Since strong resolvent convergence of the associated operators is equivalent to Mosco convergence of the corresponding positive quadratic forms
  we can rephrase Theorem \ref{construction} such that $\check{\calE}$ is the Mosco limit of $(\check{\calE}_\lambda)_{\lambda>0}$ as $\lambda$ decreases to $0$.

  {\rm(c)}
  The operator $\check{H}$ is characterized by
  \begin{align*}
    \dom \check{H} & = \bigl\{\psi\in \dom \check\calE:\; \exists\phi\in\calHaux:  \check\calE(\psi,\varphi)=(\phi,\varphi)_\aux \quad\text{for all }\varphi\in \dom \check\calE \bigr\}, \\
    \check{H}\psi & = \phi.
  \end{align*}
  Since $\ran J$ is a core for $\check\calE$, the domain of $\check H$ is also given by
  \[
  \dom \check{H} = \bigl\{\psi\in \dom\check\calE:\; \exists\phi\in\calHaux: \check\calE(\psi,Jv)=(\phi,Jv)_{\aux}  \quad\text{for all }v\in \dom J \bigr\}.
  \]

  {\rm(d)}
  If $\calHaux = \calH$, $\dom J$ dense in $\calH$ and $J$ (and hence also $J_\lambda$) is the natural embedding $J\from \dom J\ni u\mapsto u\in\calH$, then $\check{\calE} = \overline{\calE|_{\dom J}}$.
  Indeed, we then obtain $P_{\lambda} u = u$ for all $u\in \dom J$ and $\lambda>0$ and $\check{\calE}_\lambda = \calE_\lambda|_{\dom J}$ for all $\lambda>0$.
  Thus, if $\dom J$ is a core for $\calE$ then $\check{\calE} = \calE$.

  {\rm(e)}
  In case $\calE$ is a positive form, but not necessarily closed, we can first consider the closure $\overline{\calE_\reg}$ of its regular part and then apply Theorem \ref{construction} to obtain its trace.
\end{rk}

We now show that our construction agrees with the one obtained in \cite{Arendt1}.

\begin{theo}
\label{thm:constructions_agree}
  Let $a$ be a form defined by $\dom a:=\dom J$, $a[u]:=\calE[u]$ for $u\in \dom J$. Let $A$ be the operator associated with $(a,J)$ according to \cite[Theorem 3.2]{Arendt1}. Then
  $A=\cH$.
\end{theo}

\begin{proof}
  First, note that $a$ is $J$-sectorial since $\calE$ is positive, and $J(\dom J) = \ran J$ is dense in $\calHaux$ by assumption on $J$.
  Further, since $a$ is symmetric and positive, $A$ is self-adjoint and positive by \cite[Remark 3.5]{Arendt1}.

  (i)
  Let $\lambda>0$.
  Define $a_\lambda:=\calE_\lambda|_{\dom J}$. Then $a_\lambda$ is $J$-sectorial and $J(\dom a_\lambda) = \ran J$ is dense in $\calHaux$.
  Let $A_\lambda$ be the operator associated with $(a_\lambda,J)$ according to \cite[Theorem 3.2]{Arendt1}. By \cite[Remark 3.5]{Arendt1}, $A_\lambda$ is self-adjoint.
  We show that $A_\lambda = \cH_\lambda$. Indeed, let $x\in \dom \cH_\lambda \subseteq \dom \ccalE_\lambda = \ran J$. Let $u\in \dom J$ such that $Ju=x$.
  Then for $v\in \dom J$ we obtain
  \[(\cH_\lambda x,Jv)_{\aux} = \ccalE_\lambda(x,Jv) = \ccalE_\lambda(Ju,Jv) = \calE_\lambda(P_\lambda u,P_\lambda v) = \calE_\lambda(P_\lambda u,v).\]
  By Lemma \ref{lem:projection}, $P_\lambda u\in \dom J$ and $x=Ju = JP_\lambda u$.
  Define $u_n:=P_\lambda u$ for all $n\in\N$. Then $Ju_n = x$ for all $n\in\N$,
  \[\sup_{n\in\N} a_\lambda[u_n] = \sup_{n\in\N} \calE_\lambda [P_\lambda u] \leq \calE_\lambda[u],\]
  and
  \[\lim_{n\to\infty} a_\lambda(u_n,v) = \lim_{n\to\infty} \calE_\lambda(P_\lambda u,v) = (\cH_\lambda x,J(v))_{\aux}\]
  for all $v\in \dom J = \dom a_\lambda$. Hence, $x\in \dom A_\lambda$ and $A_\lambda x = \cH_\lambda x$, i.e.\ $\cH_\lambda\subseteq A_\lambda$. Since both operators are self-adjoint, they are equal.

  (ii)
  Note that $\dom a_\lambda = \dom a$ and $a_\lambda[u]-a[u] = \lambda\|u\|^2\geq 0$ for all $u\in \dom a$. Moreover, $\lim_{\lambda\to 0} a_\lambda [u] = a[u]$ for all $u\in \dom a$, and $J(\dom a_\lambda) = J(\dom J) = \ran J$ is dense in $\calHaux$.
  By \cite[Theorem 3.7]{Arendt1}, we have
  \[\lim_{\lambda\to 0} (A_\lambda+1)^{-1} = (A+1)^{-1} \quad\text{strongly}.\]

  (iii)
  By Theorem \ref{construction}, we have
  \[\lim_{\lambda\to 0} (\cH_\lambda+1)^{-1} = (\cH+1)^{-1} \quad\text{strongly}.\]
  Since strong limits are unique we obtain $A=\cH$.
\end{proof}

\begin{rk}
  {\rm(a)}
      Note that the construction in \cite{Arendt1} is valid for more general situations than we consider here; $J$ just needs to be linear with dense range and $a$ only needs to be a $J$-sectorial sesquilinear form. However, then the operator associated with $(a,J)$ is described in a somewhat implicit form in \cite[Theorem 3.2]{Arendt1}.

 {\rm(b)}
      In case $\dom J = \calD$ and $J_1$ is bounded on $(\calD,\calE_1^{1/2})$, we can also apply \cite[Theorem 4.2]{vogt} to obtain a self-adjoint operator (which is actually $\check{H}$) associated with $(\calE,J)$, and then
      obtain $\ccalE$ as the form associated with this operator. Note that since $\calE$ is symmetric and hence $J$-sectorial, the constructions in \cite[Theorem 4.2]{vogt} and \cite[Theorem 3.2]{Arendt1} agree.
\end{rk}

Next we proceed to show that our construction is consistent. We start by showing the Dirichlet principle for the form $\ccalE_0$, analogously to Theorem \ref{thm:Dirichlet_principle}.

\begin{lem}
\label{Infimum}
  Let $u\in\dom J$. Then
  \[
    \ccalE_0[Ju]=\inf\{\calE[v]:\; v\in\dom J,\ Jv=Ju\}.
  \]
\end{lem}

\begin{proof}
  By definition of $\ccalE_0$ and Theorem \ref{thm:Dirichlet_principle} we obtain
  \begin{align*}
    \ccalE_0[Ju] & = \lim_{\lambda\downarrow 0} \ccalE_\lambda[Ju] = \lim_{\lambda\downarrow 0} \inf\{\calE_\lambda[v]:\; v\in\dom J,\ Jv=Ju\} \\
    & \geq \inf\{\calE[v]:\; v\in\dom J,\ Jv=Ju\}.
  \end{align*}
  Conversely, let $v\in\dom J$ such that $Jv=Ju$. Then
  \[
    \inf\{\calE_\lambda[w]:\; w\in\dom J,\ Jw=Ju\} \leq \calE_\lambda[v]
  \]
  for all $\lambda>0$. Passing to the limit leads to $\ccalE_0[Ju]\leq \calE[v]$. Thus,
  \[
    \ccalE_0[Ju] \leq \inf\{\calE[v]:\; v\in\dom J,\ Jv=Ju\},
  \]
  which finishes the proof.
\end{proof}

The following proposition is actually a consequence of Theorem \ref{thm:constructions_agree} and \cite[Theorem 3.7]{Arendt1}. However, we shall give an independent proof.

\begin{prop}
  Let $\beta>0$. The trace of  $\calE_\beta$ is the form  $\check{\calE}_\beta$ as given by (\ref{construction1}). In other words,
  \[
    \check{\calE}_\beta=\lim_{\lambda\downarrow 0}\check{\calE}_{\beta+\lambda}.
  \]
  \label{consistency}
\end{prop}

\begin{proof}
  Let $u\in\dom J$. Applying Lemma \ref{Infimum} to $\calE_\beta$ (instead of $\calE$) and taking into account Theorem \ref{thm:Dirichlet_principle} we obtain
  \[
    \lim_{\lambda\downarrow 0}\ccalE_{\beta+\lambda}[Ju]=\inf\{\calE_\beta[v]:\;v\in\dom J,\ Jv=Ju\}=\ccalE_\beta[Ju].
  \]
  Since $\ccalE_\beta$ is closed, Theorem \ref{construction} yields $\lim_{\lambda\downarrow 0}\ccalE_{\beta+\lambda} =\ccalE_\beta$.
\end{proof}

The following result expresses the fact that some properties of the operator $\check{H}$ are strongly related to those of $J_1$.
A similar result can be found in \cite[Proposition 4.20]{vogt} (note that the corresponding construction of traces is different).

\begin{theo}
\label{compact}
  Let $\dom J = \calD$ and $J_1\from (\calD,\calE_1)\to \calHaux$ be bounded.
  
  {\rm(a)}
  Let $\ccalE_0$ be closed and $J_1$ compact. Then $\check{H}$ has compact resolvent.
  
  {\rm(b)}
  Let $\check{H}$ have compact resolvent. Then $J_1$ is compact.
\end{theo}

\begin{proof}
  (a) 
  Note that $\ccalE = \ccalE_0$ since $\ccalE_0$ is closed. In particular, $\dom\ccalE = \ran J$.
  Let us set
  \[
    (\check{\calE})_1:= \check\calE +\|\cdot\|_{\aux}^2.
  \]
  It is well known that the operator $\check{H}$ has compact resolvent if and only if the embedding $i\from (\ran J, (\check{\calE})_1^{1/2})\to\calHaux$ is compact.
  By the boundedness assumption for $J_1$ and the definition of $\check{\calE}_1$, we obtain
  \[
    \|Ju\|_{\aux}^2=\|JP_1u\|_{\aux}^2\leq \|J_1\|^2 \calE_1[P_1u]= \|J_1\|^2 \check{\calE}_1[Ju],
  \]
  for all $u\in\calD$. Thus, $(\check{\calE})_1$ is $\check{\calE}_1$-bounded. Since both $(\ran J,(\check{\calE})_1^{1/2})$ and $(\ran J,\check{\calE}_1^{1/2} )$ are Hilbert spaces, the latter inequality together with the open mapping theorem yield equivalence of the norms $(\check{\calE})_1^{1/2}$ and $\check{\calE}_1^{1/2}$. As by assumption $J_1$ is compact then according to  \cite[Satz 3.5]{Weidmann}, $J_1J_1^*=\check{H_1}^{-1}$ is compact as well, which in turn yields compactness of the embedding $(\ran J,\check{\calE}_1^{1/2})\to\calHaux$. Accordingly the embedding $(\ran J,(\check{\calE})_1^{1/2})\to\calHaux$ is compact and $\check{H}$ has compact resolvent.

  (b) 
  Note that $0\leq\ccalE\leq\check{\calE}_1$ and therefore
  \[
    0\leq (\check{H}_1 +1)^{-1}\leq (\check{H} +1)^{-1}.
  \]
  Thus, we obtain
  \[
    \| (\check{H}_1+1)^{-\frac{1}{2}}\psi\|_\aux\leq \|(\check{H}+1)^{-\frac{1}{2}}\psi\|_\aux \quad\text{for all } \psi\in \calH_\aux.
  \]
  Since $(\check{H}+1)^{-1}$ is compact, also $(\check{H}+1)^{-\frac{1}{2}}$ is compact.
  Hence, $(\check{H}_1+1)^{-\frac{1}{2}}$ is also compact, which in turn implies the compactness of $(\check{H}_1 +1)^{-1}$. Therefore, also $\check{H}_1^{-1}=J_1J_1^*$ is compact. By \cite[Satz 3.5]{Weidmann}, $J_1J_1^*$ is compact if and only if $J_1$ is compact.
\end{proof}

\begin{rk}
  We shall show in Remark \ref{J-elliptic} that the form $\ccalE_0$ is closed in case $\calE$ is $J$-elliptic. Thus, Theorem \ref{compact}(a) is a generalization of \cite[Lemma 2.7]{Arendt1}.
\end{rk}

\section{Some special situations for constructions of traces}
\label{sec:special_Situations}

In this section we provide concrete relevant situations, in which the trace form is computed explicitly.
We start with the following situation.
In many applications, especially from PDEs, it may happen that the quadratic form $\calE$ defines a scalar product on $\calD$.
For this particular situation, we shall give an explicit description of the trace form.
Let $\calD_\e$ be the  abstract completion of $\calD$ w.r.t.\ $\calE(\cdot)^{1/2}$.
Then the quadratic form $\calE$ extends in a natural way to a bounded quadratic form on the Hilbert space $(\calD_\e,\calE)$ which we still denote by $\calE$. Suppose that $\ker J$ is $\calE$-closed. Let $(\ker J)^{\perp_{\calE}}$ be the $\calE$-orthogonal complement of $\ker J$ in the Hilbert space $(\calD_\e,\calE)$ and let $P$ the $\calE$-orthogonal projection onto $(\ker J)^{\perp_{\calE}}$.
In this framework we construct a form $Q$, as before, by
\begin{equation}
  \dom Q := \ran J,\quad Q[Ju] := \calE[Pu] \quad\text{for all } u\in \dom J,
\label{trans}
\end{equation}
analogously to \eqref{construction1}.
Obviously, $Q$ is well-defined.

\begin{prop}
\label{prop-trans}
  Let $\calE$ define a scalar product on $\calD$. Assume that $\ker J$ is $\calE$-closed and let $Q$ be the quadratic form defined by \eqref{trans}.
  Then $\check{\calE} = \overline{Q_\reg}$.
  Moreover, if $J\from \dom J\subseteq (\calD_\e,\calE)\to \calHaux$ is closed then $\check{\calE} = Q$.
\end{prop}

\begin{proof}
  Since $P$ is the $\calE$-orthogonal projection onto $(\ker J)^\perp$, for $u\in \dom J$ we have $u-Pu\in \ker J$ and therefore $Pu\in \dom J$ and $Ju = JPu$.
  (This is essentially Lemma \ref{lem:projection}; there we only used that $\ker J$ is closed.)
  Consequently, the Dirichlet principle still holds true. Hence, together with Lemma \ref{Infimum} this yields for $u\in \dom J$ 
  \[Q[Ju] = \calE[Pu] = \inf\{\calE[v]:\; v\in\dom J,\ Jv=Ju\} = \ccalE_0[Ju].\]
  Thus, $Q=\ccalE_0$. By Theorem \ref{construction}, we achieve $\overline{Q_\reg}=\check\calE$.

  Now, assume that $J\from \dom J\subseteq (\calD_\e,\calE)\to \calHaux$ is closed. Then $\ker J$ is closed (w.r.t. $\calE$). Hence mimicking the proof of \cite[Theorem 1.1]{BBBMath} we conclude that the quadratic form $Q$ defined by \eqref{trans}
  is closed. Hence $Q = \overline{Q_\reg} = \check\calE$.
\end{proof}

Towards providing other situations for which an explicit computation of $\check\calE$ is still possible we introduce the vector space
\begin{align*}
  \Har:=\{u\in \dom J:\; \calE(u,v)=0 \quad \text{for all }v\in \ker J\}.
\end{align*}
Assume that $\dom J$ decomposes into  a direct sum
\begin{align*}
  \dom J=\Har\oplus \ker J.
\end{align*}
For each $u\in \dom J$ let $E_\har u$ be the unique element in $\Har$ such that
\[
  u= E_\har u + (u-E_\har u),
\]
where the decomposition is unique. Then $E_\har$ is the projection from $\dom J$ onto $\Har$ along $\ker J$, and $E_\har u$ can be interpreted as a abstract `harmonic extension' of $Ju$ for $u\in \dom J$; cf.\ Example \ref{ex:1/2-Laplacian} for a similar construction inspiring the name.
Define $\calE_\har$ in $\calHaux$ by
\[
  \dom \calE_\har :=\ran J,\quad \calE_\har[Ju] := \calE[E_\har u] \quad\text{for all } u\in \dom J.
\]
Clearly, $\calE_\har$ is then well-defined (by the direct sum assumption $u\in\ker J$ implies $E_\har u = 0$).

\begin{lem}
\label{quadharm}
  Assume $\dom J=\Har\oplus \ker J$.
  Let $u\in \dom J$. Then
  \[
    \calE(E_\har u,v) = \calE(u,E_\har v)  \quad\text{for all } v\in \dom J.
  \]
\end{lem}

\begin{proof}
  Let $v\in \dom J$. Then $\calE(E_\har u,v-E_\har v) = 0$. Thus
  \[
    \calE(E_\har u,v) = \calE(E_\har u, E_\har v) + \calE(E_\har u, v-E_\har v) = \calE(E_\har u, E_\har v).
  \]
  By symmetry,
  \[
    \calE(E_\har u,v) = \calE(E_\har u, E_\har v) = \calE(u, E_\har v). \qedhere
  \]
\end{proof}

Mimicking the proof of Proposition \ref{prop-trans} we obtain:

\begin{prop}
\label{constructiondecomp}
  Assume $\dom J=\Har\oplus \ker J$.
  Then the trace form $\check{\calE}$ coincides with the closure of the regular part of $\calE_\har$.
\end{prop}

Here is a sufficient condition for $\calE_\har$ to be closed and hence for $\check{\calE} = \calE_\har$.
\begin{lem}
\label{lem:calH_har_closed}
  Let $\dom J = \Har\oplus\ker J$.
  Assume that $(\calD,\calE)$ is a Hilbert space.
  Then $\calE_{\har}$ is closed.
\end{lem}

\begin{proof}
  Note that $(\calD,\calE_1)\ni u\mapsto u \in(\calD,\calE)$ is a contractive bijection between Banach spaces, hence has a continuous inverse. Thus, $J$ is $\calE$-closed.

  Let $(u_n)$ in $\dom J$ such that $(Ju_n)$ is a Cauchy-sequence for $\calE_{\har}$ and $Ju_n\to \tilde{u}$ in $\calHaux$ for some $\tilde{u}\in\calHaux$.
  Then
  \[\calE[E_\har u_n - E_\har u_m] = \calE_{\har}[Ju_n-Ju_m] \to 0 \quad(m,n\to\infty),\]
  so $(E_\har u_n)$ is a Cauchy-sequence for $\calE$.
  Since $(\calD,\calE)$ is a Hilbert space, there exists $u\in \calD$ such that $E_\har u_n\to u$ in $\calE$. Note that $E_\har u = u$.
  Since $Ju_n = JE_\har u_n$ for all $n\in\N$ and $J$ is $\calE$-closed, we obtain $u\in \dom J$ and $Ju = \tilde{u}$.
  Hence, $\tilde{u} = Ju\in \dom \calE_{\har}$ and $\calE_{\har}[Ju_n-Ju] = \calE[E_\har u_n-E_\har u] = \calE[E_\har u_n-u]\to 0$.
  Thus, $\calE_{\har}$ is closed.
\end{proof}

For an application of the situation in Lemma \ref{lem:calH_har_closed} see \cite{FreibergSeifert2015, SeifertVoigt2011}.

By means of Proposition \ref{constructiondecomp} we can now handle the following case.
Assume that $\ker J$ is dense in $\calH$ and define the form $\calE_\Dir$ in $\calH$ by
\[
  \dom \calE_\Dir := \ker J,\quad \calE_\Dir[u] := \calE[u] \quad\text{for all } u\in \dom \calE_\Dir
\]
Then $\calE_\Dir$ is closed. Indeed, let $(u_n)$ in $\dom \calE_\Dir$, $u\in\calH$, $\calE_\Dir[u_n-u_m]\to 0$ ($m,n\to\infty$), $u_n\to u$ in $\calH$.
Then $\calE[u_n-u_m]\to 0$. Since $\calE$ is closed, we obtain $u\in \calD$ and $\calE[u_n-u]\to 0$. Since $J_1$ is closed and $Ju_n = 0\to 0$ we obtain $u\in \ker J = \dom \calE_\Dir$. Thus,
\[
  \calE_\Dir[u_n-u] = \calE[u_n-u]\to 0.
\]
Let $L_\Dir$ be the positive self-adjoint operator associated with $\calE_\Dir$.
Assume that
\[
  \Har\cap \ker J=\{0\}.
\]
Then $L_\Dir$ is injective. Indeed, let $u\in \dom L_\Dir\subseteq \ker J \subseteq \calD$ such that $L_\Dir u = 0$. Then, for $v\in \ker J$ we obtain
\[
  \calE(u,v) = (L_\Dir u, v) = 0.
\]
Thus, $u\in \Har$, and therefore $u\in\Har\cap \ker J = \{0\}$.
Assume that $L_\Dir$ is surjective.
For $u\in\calD$ and $\lambda>0$ set
\[
  v_\lambda := \lambda L_\Dir^{-1}P_\lambda u + P_\lambda u.
\]

\begin{lem}
\label{indep}
  Let $\ker J$ be dense in $\calH$, $\Har\cap \ker J=\{0\}$ and $L_\Dir$ surjective.
  
  {\rm(a)}
      Let $u\in \dom J$, $\lambda>0$. Then $v_\lambda\in\Har$ and $Jv_\lambda=Ju$.
      Furthermore, $v_\lambda$ is $\lambda$-independent.
      
  {\rm(b)}
      $\dom J = \Har\oplus \ker J$. Moreover, $E_\har u = v_1$ for $u\in \dom J$.
\end{lem}

\begin{proof}
  Since $\ran L_\Dir^{-1} = \dom L_\Dir \subseteq \ker J$ and $P_\lambda u\in \dom J$ for all $u\in\dom J$ by Lemma \ref{lem:projection}, we get $v_\lambda\in \dom J$ and $Jv_\lambda=JP_\lambda u=Ju$.

  (a)
  Let $v\in \ker J$ ($=\dom \calE_\Dir$). Then $P_\lambda u \bot_{\calE_\lambda} v$ and therefore
  \begin{align*}
    \calE(v_\lambda,v) & = \lambda\calE(L_\Dir^{-1}P_\lambda u,v) + \calE(P_\lambda u,v)
    = \lambda\calE_\Dir(L_\Dir^{-1}P_\lambda u,v) - \lambda(P_\lambda u,v) \\
    & = \lambda(P_\lambda u,v)-\lambda(P_\lambda u,v) = 0.
  \end{align*}
  Thus, $v_\lambda\in\Har$.
  Hence, for $\lambda,\lambda'>0$ we obtain $v_\lambda-v_{\lambda'}\in \Har\cap \ker J=\{0\}$.

  (b)
  Let $u\in \dom J$. Then $u = P_\lambda u + u-P_\lambda u = v_\lambda - \lambda L_\Dir^{-1}P_\lambda u + u-P_\lambda u$,
  where $v_\lambda \in \Har$ and $u-v_\lambda=-\lambda L_\Dir^{-1}P_\lambda u + u-P_\lambda u\in \ker J$. Hence, making use of assertion (a) we obtain $\dom J = \Har\oplus \ker J$.

  Observing that $u=v_1 + (u-v_1)\in \Har\oplus \ker J$ yields $E_\har u = v_1$ by definition of $E_\har$.
\end{proof}

\begin{prop}
\label{constharm}
  Let $\ker J$ be dense in $\calH$, $\Har\cap \ker J=\{0\}$ and $L_\Dir$ surjective.
  Then $\check{\calE}=\overline{(\calE_\har)_\reg}$.
\end{prop}

\begin{proof}
  By Lemma \ref{indep} we have $\dom J = \Har\oplus \ker J$. Now the result follows from Proposition \ref{constructiondecomp}.
\end{proof}

\begin{rk}
    Proposition \ref{constharm} is inspired from the construction of the trace of the quad\-ratic form associated with the Neumann-Laplacian
    on bounded open subsets of $\R^n$ with Lipschitz boundary (see e.g.\ \cite{Daners}).
    The corresponding operator is then the Dirichlet-to-Neumann operator.
\end{rk}

As a next step we shall give another general case where $\dom J=\Har\oplus\ker J$ is fulfilled
and hence Proposition \ref{constructiondecomp} can be applied.
Consider the form $\calE^{J}$ in $\calH$ defined by
\[
  \dom \calE^J := \dom J,\quad \calE^J[u] :=\calE[u]+\|Ju\|_{\aux}^2 \quad\text{for all } u\in \dom \calE^J.
\]
Then
\[
  \Har=\{u\in \dom J:\; \calE^J(u,v)=0 \quad\text{for all } v\in \ker J\}.
\]
Assume that $\calE^J$ defines a scalar product on $\dom J$.
Let us denote by $\calH^J$ the $\calE^J$-completion of $\dom J$ and by $P^J$ the $\calE^J$-orthogonal projection onto the $\calE^J$-orthogonal complement of $\ker J$.

\begin{prop}
\label{prop:calE_har}
  Assume that $\calE^J$ defines a scalar product on $\dom J$ and $\ker J$ is $\calE^J$-closed.
  Then $\dom J=\Har\oplus \ker J$. Moreover,
  \[
  \calE_\har[Ju] = \calE[P^Ju] \quad\text{for all } u\in \dom J.
  \]
\end{prop}

\begin{proof}
  By  assumption we obtain $\Har\cap\ker J=\{0\}$.
  Thus, we have to show that every $u\in \dom J$ admits a decomposition.
  Let $u\in \dom J$. As, by assumption, $\ker J$ is $\calE^J$-closed we obtain $u-P^J u\in\ker J\subseteq \dom J$. Consequently, $P^Ju\in\dom J$. Hence, $P^Ju\in\Har$ and $u=P^Ju + u-P^Ju$. Therefore, we obtain $\dom J=\Har\oplus \ker J$.

  It remains to prove that $P^Ju=E_\har u$ for all $u\in \dom J$.
  Let $u\in \dom J$. Then $u$ admits a unique decomposition $u = E_\har u + (u-E_\har u)$ with $E_\har u\in \Har$. Since $u = P^Ju + (u-P^Ju)$ with $P^J u\in\Har$, we observe $E_\har u = P^Ju$.
\end{proof}

\begin{rk}\label{J-elliptic}
  {\rm(a)}
  Assume that $\calE$ is $J$-elliptic, i.e.\ $J$ is everywhere defined and bounded on $\calD$ and there exist $\beta\in\R,\ \alpha>0$ such that
  \[
  \calE[u] + \beta\|Ju\|_{\aux}^2\geq \alpha\calE_1[u] \quad\text{for all }u\in\calD.
  \]
  Then $\calE^J$ yields a scalar product on $\calD$ and $J$ is $\calE^J$-closed.
  Thus, $\ker J$ is $\calE^J$-closed.
  Hence, applying Proposition \ref{prop:calE_har} and then Proposition \ref{constructiondecomp} we obtain $\ccalE = \overline{(\calE_\har)_{\reg}}$.
  Moreover, a straightforward computation shows that the form $\calE_\har$ is closed, and therefore 
  $\ccalE = \ccalE_0 = \calE_\har$.

  {\rm(b)}
  Assume that $\Har\cap \ker J=\{0\}$. Then $\calE^J$ defines a scalar product on $\dom J$.
  Indeed, for $u\in \dom J$ with $\calE^J[u] = 0$ we obtain $u\in \ker J$ and $\calE[u] = 0$.
  Hence, by the Cauchy-Schwarz inequality
  \[
    |\calE(u,v)| \leq \calE[u]^{1/2} \calE[v]^{1/2} = 0
  \]
  for all $v\in \ker J$ and therefore $u\in \Har$. Hence, $u=0$.
\end{rk}
\section{Examples}
\label{sec:examples}

In this section we work out some examples to illustrate our method for constructing traces of forms.

\begin{exa}
  Let $\Omega,\Omega_0\subseteq \R^d$ be open and bounded with boundaries
  $\Gamma:=\partial \Omega$ and $\Gamma_0:=\partial\Omega_0$ such that
  $\overline{\Omega_0}\subseteq \Omega$. Assume that $\Gamma_0, \Gamma$ are $C^1$.
  Consider the quadratic form $\calE$ in $L^2(\Om)$ given by
  \[
    \calD := H_0^1(\Omega),\quad \calE[u]:=\int_{\Om} |\nabla u|^2,
  \]
  and let $J\colon \calD\to L^2(\Omega_0)$, $Ju:=u|_{\Omega_0}$.
  Then $(\calD,\calE)$ is a Hilbert space.
  Thus, we can construct $\check{\calE}$ by means of Proposition \ref{prop-trans}.
  Let $P$ be the $\calE$-orthogonal projection onto the $\calE$-orthogonal complement of $\ker J$.
  Obviously,
  \[
  (\ker J)^{\perp_\calE} = \{u\in H_0^1(\Omega): \Delta (u|_{\Omega\setminus \overline{\Omega_0}}) = 0\},
  \]
  and for each $u\in H_0^1(\Omega)$ we have that $Pu$ is the unique element in
  $H^1_0(\Omega)$ such that $\Delta (Pu|_{\Omega\setminus\overline{\Omega_0}})=0$ and $Pu|_{\Omega_0}=u|_{\Omega_0}$.
  The trace form $\check{\calE}$ is given by
  \[
    \dom \check{\calE} = \ran J = H^1(\Omega_0), \quad
    \check{\calE}(Ju,Jv) = \calE(Pu,Pv) = \calE(Pu,v) \quad\text{for all } u,v\in H_0^1(\Omega).
  \]
  Applying Green's formula, we derive
  \begin{align*}
    \check{\calE}(Ju,Jv) & = \int_{\Omega_0} \nabla Pu\cdot \nabla v + \int_{\Omega\setminus\Omega_0} \nabla Pu\cdot \nabla v \\
    & = \int_{\Omega_0} \nabla u \cdot \nabla v + \langle \gamma_1^- Pu,\gamma_0 v\rangle_{H^{-1/2}(\Gamma_0),H^{1/2}(\Gamma_0)},
  \end{align*}
  where $\gamma_1^- Pu \in H^{-1/2}(\Gamma_0)$ is the conormal derivative of $Pu|_{\Omega\setminus\overline{\Omega}_0}$ on $\Gamma_0$ and $\gamma_0v\in H^{1/2}(\Gamma_0)$ is the trace of $v$ on $\Gamma_0$; cf.\ \cite[Section 5.5.1]{DemengelDemengel2012}.
  Note that if $Pu|_{\Omega\setminus\overline{\Omega_0}}\in H^2(\Omega\setminus\overline{\Omega_0})$ the linear
  functional $\gamma_1^- Pu$ on $H^{1/2}(\Gamma_0)$ coincides with
  the strong conormal derivative $\partial_\nu Pu|_{\Gamma_0} = \nabla Pu|_{\Omega\setminus\overline{\Omega}_0} \cdot \nu\in L^2(\Gamma_0)$,
  where $\nu$ is the outward unit normal on $\Gamma_0$ (with respect to $\Omega\setminus\overline{\Omega_0}$).

  For $u\in H^1(\Omega_0)$ such that $\Delta u\in L^2(\Omega_0)$ we set $Pu:=P\tilde{u}$, where
  $\tilde{u}$ is any extension of $u$ in $H^1_0(\Omega)$ (for the existence of such an extension see e.g.\ \cite[Proposition 2.70]{DemengelDemengel2012}),
  and let $\gamma_1^+u \in H^{1/2}(\Gamma_0)$ be the conormal derivative of $u$ on $\Gamma_0$ (with respect to $\Omega_0$).
  Let $\check{H}$ be the positive self-adjoint operator
  associated with $\check{\mathcal{E}}$. Then
  \[
    \dom \check{H} = \{u\in H^1(\Omega_0):\; \Delta u\in L^2(\Omega_0),\, \gamma_1^+ u + \gamma_1^- Pu=0\},\quad
    \check{H}u = -\Delta u.
  \]
  Indeed, note that for $f\in L^2(\Omega_0)$ we have $u\in \dom \check{H}$
  and $\check{H}u=f$ if and only if
  \[
    \check{\mathcal{E}}(u,v) = (f,v)_{L^2(\Omega_0)} \quad\text{for all } v\in H^1(\Omega_0).
  \]
  Let $u\in \dom \check{H}$.
  By taking $v\in C_c^\infty(\Omega_0)$ we obtain $\check{H}u=-\Delta u$.
  Green's formula yields
  \[
    \int_{\Omega_0} \nabla u \cdot \nabla v = (-\Delta u,v)_{L^2(\Omega_0)} + \langle \gamma_1^+u,\gamma_0 v\rangle_{H^{-1/2}(\Gamma_0),H^{1/2}(\Gamma_0)}.
  \]
  Thus, $\gamma_1^+ u+\gamma_1^- Pu=0$ in $H^{-1/2}(\Gamma_0)$.

  Conversely, if $u\in H^1(\Omega_0)$ such that $\Delta u \in L^2(\Omega_0)$ and $\gamma_1^+ u+\gamma_1^- Pu=0$ in $H^{-1/2}(\Gamma_0)$,
  then for all $v\in H^1(\Omega_0)$ we obtain
  \begin{align*}
    \int_{\Omega_0} \nabla u\cdot \nabla v & =  (-\Delta u,v)_{L^2(\Omega_0)} + \langle \gamma_1^+u,\gamma_0 v\rangle_{H^{-1/2}(\Gamma_0),H^{1/2}(\Gamma_0)} \\
    & = (-\Delta u,v)_{L^2(\Omega_0)} - \langle \gamma_1^- Pu,\gamma_0 v\rangle_{H^{-1/2}(\Gamma_0),H^{1/2}(\Gamma_0)},
  \end{align*}
  and therefore
  \[
    \check{\calE}(u,v) = \int_{\Omega_0} \nabla u\cdot \nabla v + \langle \gamma_1^- Pu,\gamma_0 v\rangle_{H^{-1/2}(\Gamma_0),H^{1/2}(\Gamma_0)} = (-\Delta u,v)_{L^2(\Omega_0)}.
  \]
  Thus, $u\in \dom \check{H}$ and $\check{H}u=-\Delta u$.

  Since the boundary of $\Om$ is of class $C^1$ and $\Omega_0$ is bounded, by Rellich-Kondrachov Theorem the embedding $(\calD,\calE_1^{1/2})\to L^2(\Omega_0)$ is compact.
  By Theorem \ref{compact} we obtain that $\check{H}$ has compact resolvent.
\end{exa}

Next, we revisit the $1/2$-Laplacian, see \cite{Caffarelli}.

\begin{exa}
\label{ex:1/2-Laplacian}
  Let $d\in\N$ and $\R^{d+1}_+:=\R^d\times (0,\infty)$.
  Let $\calH:=L^2(\R^{d+1}_+)$, $\calHaux:=L^2(\R^d)$, and define $\calE$ in $\calH$ by
  \[
    \calD:=H^1(\R^{d+1}_+),\quad \calE[u]:=\int_{\R^{d+1}_+}|\nabla u(x,t)|^2\,dx\,dt.
  \]
  Let $J\from\calD\to \calHaux$ be defined by $Ju:=\gamma_0 u$, where $\gamma_0$ is the trace of $u$ on the boundary of $\R^{d+1}_+$.
  Then $J$ is bounded on $(\calD,\calE_1^{1/2})$ and $\ran J= H^{1/2}(\R^d)$ is dense in $\calHaux$.
  Let $\psi\in\ran J$, $\lambda>0$. Let $u\in\calD$ such that $Ju=\psi$. Then $P_\lambda u$ is the unique element in $H^1(\R^{d+1}_+)$ which solves the boundary value problem
  \begin{align*}
    -\Delta P_\lambda u+\lambda P_\lambda u & = 0 \quad\text{ in }\R^{d+1}_+, \\
    P_\lambda u & = \psi  \quad\text{ on }\R^{d}.
  \end{align*}
  Thus, by Fourier transform with respect to the variable $x$ we obtain an ordinary differential equation
  \begin{align*}
    |\xi|^2 \widehat{P_\lambda u}(\xi,t) - {\frac{\partial^2 \widehat{P_\lambda u}}{\partial^2 t}}(\xi,t)
    +\lambda \widehat{P_\lambda u}(\xi,t) & = 0  \quad\text{ for } (\xi,t)\in \R^d\times (0,\infty), \\
    \widehat{P_\lambda u}(\xi,0) & = \hat{\psi}(\xi) \quad\text{ for } \xi\in \R^d.
  \end{align*}
  The solution is given by
  \[
    \widehat{P_\lambda u}(\xi,t)= e^{-\sqrt{|\xi|^2 +\lambda} \,t}\hat{\psi}(\xi).
  \]
  Hence,
  \begin{align*}
  \check{\calE}_\lambda[\psi] & = \calE_\lambda[P_\lambda u]
    = \int_0^\infty\int_{\R^d} |\nabla P_\lambda u(x,t)|^2\,dx\,dt + \lambda \int_0^\infty\int_{\R^d} |P_\lambda u(x,t)|^2\,dx\,dt \\
  & = \int_0^\infty\int_{\R^d} |\xi|^2|\widehat{P_\lambda u}(\xi,t)|^2\,d\xi\,dt + \int_0^\infty\int_{\R^d}  \Bigl|{\frac{\partial \widehat{P_\lambda u}}{\partial t}}(\xi,t)\Bigr|^2\,d\xi\,dt \\
  & + \lambda\int_0^\infty\int_{\R^d} |\widehat{P_\lambda u}(\xi,t)|^2\,d\xi\,dt.
  \end{align*}
  Using Fubini's Theorem and an integration by parts for the second integral in the latter identity we thus obtain
  \[
   \check{\calE}_\lambda[\psi] = \int_{\R^d}\sqrt{|\xi|^2 +\lambda}\,|\hat\psi(\xi)|^2\,d\xi \to \int_{\R^d}|\xi||\hat\psi(\xi)|^2\,d\xi.
  \]
  One can easily check that the limiting quadratic form is closed. Hence, from Theorem \ref{construction} we observe that $\check{\calE} = \lim_{\lambda\downarrow 0}\check{\calE}_\lambda$,
  which is nothing else but the closed positive form associated with $(-\Delta)^{1/2}$ on $\R^d$.
\end{exa}

\begin{exa}
  Let $\calE$ be the classical Dirichlet form in $L^2(\R)$, i.e.
  \[
    \calD := H^1(\R),\quad \calE(u,v) := \int u' \overline{v}' \quad\text{for all }u,v\in H^1(\R).
  \]
  Let $(a_n)_{n\in\Z}$ be a sequence in $(0,\infty)$ and
  $\mu:=\sum_{n\in\Z}a_n\delta_n$. By Sobolev's embedding theorem, every $u\in H^1(\R)$ has a unique continuous representative $\tilde{u}$.
  We shall assume that every element in $H^1(\R)$ is continuous. We define the operator $J$ from
  $\calD$ to $L^2(\R,\mu)$ by
  \[
    \dom J:=\{u\in H^1(\R):\; \sum_{n\in\Z} a_n |u(n)|^2 < \infty\},\quad Ju:=u|_{\Z} \quad\text{for all }u\in \dom J.
  \]
  Then $J$ is densely defined in $(\calD,\calE_1^{1/2})$ and the range of $J$ is dense in $L^2(\R,\mu)$.
  Moreover, $J$ is everywhere defined on $\calD$ and bounded on $(\calD,\calE_1^{1/2})$ if and only if $(a_n)$ is bounded.
  We claim that the operator $J$ is closed in $(\calD,\calE_1^{1/2})$. Indeed, let $(u^k)_k$ be a sequence in $\dom J$ such that $(u^k)_k$ converges to $u$ in
  $(\calD,\calE_1^{1/2})$ and
  $(Ju^k)_k$ converges to $v$ in $L^2(\R,\mu)$. Then, by Sobolev's inequality, the sequence $({u}^k)_k$ converges locally uniformly (and therefore pointwise) to $u$.
  Thus, $u=v$ $\mu$-a.e., yielding $u\in \dom J$ and $Ju=v$.

  For every $\lambda>0$ we obtain
  \begin{align*}
    \dom \check{\calE}_\lambda & = \ran J = \bigl\{\psi\in L^2(\R,\mu):\; \sum_{n\in\Z}|\psi(n)|^2<\infty \bigr\},\\
    \check{\calE}_\lambda[\psi] & = \frac{\sqrt{\lambda}}{\sinh\sqrt{\lambda}}\sum_{n\in\Z}|\psi(n+1)-\psi(n)|^2
    + 2\sqrt{\lambda}\frac{\cosh\sqrt{\lambda} - 1}{\sinh\sqrt{\lambda}} \sum_{n\in\Z} |\psi(n)|^2.
  \end{align*}
  Indeed, let $u\in H^1(\R)$. By Sobolev's inequality, applied on the intervals $(n-1/2,n+1/2)$, we obtain $\sum_{n\in\Z}|u(n)|^2<\infty$.
  Conversely, let $\psi\in L^2(\R,\mu)$ such that $\sum_{n\in\Z}|\psi(n)|^2<\infty$. Choose $\varphi\in C_c^\infty(\R)$ such that $\varphi(0) = 1$ and $\varphi(x)
  = 0$ if $|x|>1/2$. Then $\bigl(\psi(n)\varphi(\cdot-n)\bigr)_{n\in\Z}$ is an orthogonal system in $H^1(\R)$ and
  \[
    \sum_{n\in\Z}\|\psi(n)\varphi(\cdot-n)\|^2_{H^1(\R)}
    =\sum_{n\in\Z}|\psi(n)|^2\|\varphi\|^2_{H^1(\R)}<\infty.
  \]
  Thus $u:=\sum_{n\in\Z}\psi(n)\varphi(\cdot-n)\in
  H^1(\R)$. Since $u=\psi$ $\mu$-a.e., we get
  $\psi\in\ran J = \dom \check{\calE}_\lambda$. Thus,
  $\dom \check{\calE}_\lambda = \bigl\{\psi\in L^2(\R,\mu):\; \sum_{n\in\Z}|\psi(n)|^2<\infty\bigr\}$.
  Obviously,
  \[
    (\ker J_\lambda)^{\perp_{\calE_\lambda}}=\{u\in H^1(\R):\; -u''+\lambda u=0 \quad\text{in }\R\setminus\Z\}.
  \]
  Hence, for $u\in H^1(\R)$, we observe that $P_\lambda u$ is the unique element in $H^1(\R)$ such that
  \begin{align*}
    -(P_\lambda u)''+\lambda P_\lambda u & = 0 \quad\text{ in } \R\setminus\Z,\\
    P_\lambda u & = u \quad\text{ on } \Z.
  \end{align*}
  An elementary computation yields
  \begin{align*}
    P_\lambda u & = \frac{1}{\sinh\sqrt{\lambda}}\Bigl(u(n+1)\sinh\bigl(\sqrt{\lambda}(\cdot-n)\bigr)-u(n)
    \sinh\bigl(\sqrt{\lambda}(\cdot-n-1)\bigr)\Bigr)\quad\text{in } [n,n+1].
  \end{align*}
  For every $u\in \dom J$ we have
  \begin{align*}
    \check{\calE}_\lambda[Ju] & = \calE_\lambda[P_\lambda u] = \sum_{n\in\Z} \Bigl( \int_n^{n+1}|(P_\lambda u)'(x)|^2\,dx + \lambda\int_n^{n+1}|P_\lambda u(x)|^2\,dx\Bigr).
  \end{align*}
  Integrating by parts, we obtain
  \begin{align*}
    \check{\calE}_\lambda(Ju,Ju) & = \sum_{n\in\Z} P_\lambda u (\overline{P_\lambda u})' \bigr|_n^{n+1} \\
    & = \frac{\sqrt{\lambda}}{\sinh\sqrt{\lambda}} \sum_{n\in\Z} \bigl(- u(n+1)\overline{u(n)} - u(n)\overline{u(n+1)}\bigr) + 2\sqrt{\lambda}\frac{\cosh\sqrt{\lambda}}{\sinh\sqrt{\lambda}} \sum_{n\in\Z} |u(n)|^2 \\
    & = \frac{\sqrt{\lambda}}{\sinh\sqrt{\lambda}} \sum_{n\in\Z} |u(n+1)-u(n)|^2 + 2\sqrt{\lambda}\frac{\cosh\sqrt{\lambda} - 1}{\sinh\sqrt{\lambda}} \sum_{n\in\Z} |u(n)|^2.
  \end{align*}
  Letting $\lambda\downarrow 0$ we obtain
  \[
    \lim_{\lambda\downarrow 0}\check{\calE}_\lambda[Ju] = \sum_{n\in\Z}|u(n+1)-u(n)|^2 .
  \]
  The latter form is closable. Let $Q$ be the form defined by
   \[
      \dom Q = \bigl\{\psi\in L^2(\R,\mu):\; \sum_{n\in\Z}|\psi(n+1) - \psi(n)|^2<\infty\bigr\},\quad Q[\psi] = \sum_{n\in\Z}|\psi(n+1)-\psi(n)|^2.
  \]
  Then $\check\calE$ is a closed restriction of $Q$. Moreover if $\sum_n a_n=\infty$ then $\check\calE=Q$. In fact, in the latter case $\check\calE$ is the quadratic form associated with the (Neumann) graph Laplacian on the graph $\Z$ with measure determined by the sequence $(a_n)$, see e.g.\ \cite[Theorem 6]{KellerLenz2012}.\\
  Note that the sequence $(a_n)$ appears in $\check{\calE}$ only in an implicit way. In fact, it describes the measure of the space $\calHaux = L^2(\R,\mu)$, where the trace form, the form associated with the graph Laplacian, is defined on.
\end{exa}

More examples concerning singular diffusion can be found in \cite{SeifertVoigt2011, FreibergSeifert2015}.

\section{Traces of Dirichlet forms}
\label{sec:Traces_Dirichlet_forms}

In this section let $X$ be a locally compact separable metric space,
$m$ a positive Radon measure with full support $X$ and $\mu$ a positive Radon measure on $X$.
We set $\calH := L^2(X,m)$ and $\calHaux:=L^2(X,\mu)$ and assume that $\calE$ is a regular Dirichlet form in $L^2(X,m)$ with domain $\calD$.
Furthermore, let us assume that $\mu$ does not charge any sets of zero capacity.

It is well-known (see \cite[Theorem 2.1.3]{Fukushima}) that every element from the domain of a regular Dirichlet form possesses a quasi-continuous representative.
Moreover, two quasi-continuous representatives which coincide $m$-a.e.\ coincide quasi-everywhere and hence $\mu$-a.e.\ (see \cite[Lemma 2.1.4]{Fukushima}).
From now on we assume that all elements from $\calD$ are quasi-continuous.
Let $J\from \dom J:=\calD\cap L^2(X,\mu)\to L^2(X,\mu)$, $Ju:=u$. Then $J$ is well-defined.

\begin{lem}
\label{lem:properties_of_J}
    $J$ is densely defined, has dense range, and $J_1$ is closed.
\end{lem}

\begin{proof}
  Clearly, $C_c(X)\cap\calD\subseteq \dom J$. Since $\calE$ is regular, $\dom J$ dense in $(\calD,\calE_1)$ and since $\calE$ is densely defined it is also dense in $\calH$.

  Since $\calE$ is regular, $C_c(X)\cap \calD$ is dense in $C_c(X)$ (with respect to the uniform norm), which itself is dense in $L^2(X,\mu)$. Hence, it is also dense in $L^2(X,\mu)$. Since it is a subspace of $\ran J$, $J$ has dense range.

  Let $(u_n)$ in $\dom J$, $u\in\calD$ and $v\in L^2(X,\mu)$ such that $\lim_{n\to\infty}\calE_1[u_n-u]=0$ and $Ju_n\to v$.
  By \cite[Theorem 2.1.4]{Fukushima} there exists a subsequence $(u_{n_k})$ such that $u_{n_k}\to u$ q.e.\ and hence also $\mu$-a.e. Hence, $v=u$ $\mu$-a.e. and therefore $u\in \dom J$ and $Ju = u = v$.
\end{proof}

Thus, we can construct the trace of $\calE$ w.r.t.\ to $J$ as in Theorem \ref{construction}, which we still denote by $\check{\calE}$.

\begin{theo}
  The trace form $\check{\calE}$ is a Dirichlet form.
\end{theo}

\begin{proof}
  We first show that $\check{\calE}_\lambda$ is a Dirichlet form for every $\lambda>0$. We already know that $\check{\calE}_\lambda$ is densely defined and closed. Thus, to prove that it is in fact a Dirichlet form it remains to show that the unit contraction operates on $\check{\calE}_\lambda$.
  Let $u\in \dom J$. Then $(0\vee u)\wedge 1\in \calD\cap L^2(X,\mu)=\dom J$ and $(0\vee Ju)\wedge 1=J\bigl((0\vee u)\wedge 1\bigr)\in \ran J=\dom \check{\calE}_\lambda$.
  Furthermore, using the Dirichlet principle in Theorem \ref{thm:Dirichlet_principle} together with the fact that $\ccalE_\lambda$ is a Dirichlet form, we obtain
  \begin{align*}
    \check{\calE}_\lambda[(0\vee Ju)\wedge 1] & = \inf\{ \calE_\lambda[v]:\; v\in \dom J,\, Jv= J\bigl((0\vee u)\wedge 1\bigr)\} \\
    & \leq \inf\{ \calE_\lambda[(0\vee v)\wedge 1]:\; v\in \dom J,\, Jv= Ju\} \\
    & \leq \inf\{ \calE_\lambda[v]:\; v\in \dom J,\, Jv= Ju\} = \check{\calE}_\lambda[Ju].
  \end{align*}
  Thus $\check{\calE}_\lambda$ is a Dirichlet form.

  Note that $\check{\calE}$ is densely defined.
  According to \cite[Theorem 1.4.1]{Fukushima}, proving that $\check{\calE}$ is a Dirichlet form is equivalent to prove that the operator $\alpha(\check{H}+\alpha)^{-1}$ is Markovian for every $\alpha>0$.
  Let $\psi\in L^2(X,\mu)$ such that $0\leq \psi\leq 1$ $\mu$-a.e. Owing to the fact that $\check{\calE}_\lambda$ is a Dirichlet form for every $\lambda>0$, for every $\alpha>0$ we have
  \[
    0\leq \alpha(\check{H_\lambda} + \alpha)^{-1}\psi\leq 1 \quad \text{$\mu$-a.e.}
  \]
  Since $\alpha(\check{H_\lambda} + \alpha)^{-1}\to \alpha(\check{H} + \alpha)^{-1}$ strongly, also $\alpha(\check{H} + \alpha)^{-1}$ is Markovian.
\end{proof}

Let $F$ be the topological support of the measure $\mu$. If we consider $\check{\calE}$ as a Dirichlet form in $L^2(F,\mu)$ we can get more information on it.

\begin{prop}
\label{prop:checkcale_regular}
  The Dirichlet form $\check{\calE}$ considered in $L^2(F,\mu)$ is regular.
\end{prop}

\begin{proof}
  We first show that $\check{\calE}_\lambda$ is regular for every $\lambda>0$.
  Let $\lambda>0$, $\psi\in C_c(F)$. By Tietze's extension theorem, the function $\psi$ has an extension $\tilde{\psi}\in C_c(X)$.
  Since $\calE$ is regular, by \cite[Lemma 1.4.2-ii, p.29]{Fukushima} there is a sequence $(u_k)$ in $C_c(X)\cap\calD$ such that
  $\supp(u_k)\subseteq\supp(\tilde{\psi})$ for all $k\in\N$ and $\|u_k-\tilde{\psi}\|_\infty\to 0$.
  Hence, $(Ju_k)$ in $C_c(F)\cap \ran J$ and $Ju_k \to \psi$ uniformly on $F$.
  Now let $\psi\in\ran J$. Then there exists $u\in L^2(X,\mu)\cap\calD$ such that $\psi=Ju$.
  The regularity of $\calE$ and the fact that $\mu$ is a Radon measure yield the regularity of the Dirichlet form $\calE^J$ on $L^2(X,m)$ defined by
  \[
    \dom \calE^J := L^2(X,\mu)\cap\calD,\quad \calE^J[u] := \calE[u] + \int_X u^2\,d\mu \quad\text{for all }u\in \dom \calE^J,
  \]
  see \cite[Theorem 6.1.2]{Fukushima}.
  Thus, there exists a sequence $(u_k)$ in $C_c(X)\cap\calD$ such that $\calE_\lambda^J[u_k-u]\to 0$.
  Therefore, $(Ju_k)$ in $C_c(F)\cap\ran J$ and $Ju_k\to Ju$ in $L^2(F,\mu)$. By construction of $\check{\calE}_\lambda$ we obtain
  \[
    \check{\calE}_\lambda[Ju_k - \psi] = \check{\calE}_\lambda[Ju_k - Ju] = \calE_\lambda[P_\lambda u_k - P_\lambda u] \leq \calE_\lambda[u_k - u] \leq \calE^J_\lambda [u_k-u] \to 0.
  \]
  Hence, $\check{\calE}_\lambda$ is regular.

  Let us now prove the regularity of $\check{\calE}$.
  As $\ran J\subseteq \check{\calD}$, by the first part of the proof we get that $C_c(F)\cap\check{\calD}$ is uniformly dense in $C_c(F)$.
  Note that $\ran J$ is a core for $\ccalE$. Thus, it suffices to prove that $C_c(F) \cap \ran J$ is a core for $\ran J$. Let $\psi\in\ran J$. Since $\ccalE_1$ is regular, there exists a sequence $(\psi_k)_{k\in\N}$ in $C_c(F)\cap \ran J$ such that
  $\bigl(\ccalE_1\bigr)_1[\psi_k-\psi] \to 0$. Therefore,
  \[\bigl(\ccalE\bigr)_1[\psi_k-\psi] \leq \bigl(\ccalE_1\bigr)_1[\psi_k-\psi] \to 0.\]
  Hence, $\check{\calE}$ is regular.
\end{proof}

Next, we will establish a formula for $\check{H}_1^{-1}$ in terms of the $1$-potential.

\begin{lem}
  Assume that $J_1$ is bounded.
  Then for every $\psi\in L^2(X,\mu)$, the signed measure $\psi\mu$ has finite energy integral. Let $U_1^\mu\psi$ be the $1$-potential of the signed measure $\psi\mu$. Then
  \[
    \check{H}_1^{-1}\psi=J_1U_1^\mu\psi.
  \]
\label{potential}
\end{lem}

\begin{proof}
  Let us first observe that for every fixed $\psi\in L^2(X,\mu)$ the signed measure $\psi\mu$ has finite energy integral,
  i.e.\ there exists $c\geq 0$ such that
  \[
  \int |Jv \cdot \psi|\,d\mu \leq c(\calE_1[v])^{1/2} \quad\text{for all } v\in\calD.
  \]
  Thus, the $1$-potential of $\psi\mu$ is well-defined and is characterized as being the unique element from $\calD$ such that
  \[
    \calE_1(U_1^\mu\psi,v) = \int Jv\cdot\psi\,d\mu \quad\text{for all } v\in\calD.
  \]
  Hence, making use of the construction of $\check{\calE}_1$ together with the latter identity we obtain
  \begin{align*}
    \ccalE_1(J_1U_1^\mu\psi,Jv) & = \calE_1(U_1^\mu \psi,P_1v) = \int JP_1v\cdot\psi\,d\mu
    = \int Jv\cdot\psi\,d\mu \\
    & = (\psi,Jv)_{L^2(X,\mu)}\quad\text{for all } v\in\calD,\ \psi\in L^2(X,\mu).
  \end{align*}
  Thus $J_1U_1^\mu\psi\in D(\check{H}_1)$ and $\check{H}_1 J_1U_1^\mu\psi = \psi$.
\end{proof}

We end this section by showing that our construction of the trace of a Dirichlet form coincides with the construction in \cite[Section 6.2]{Fukushima}.
To this end, let
\begin{align*}
  \calD_\e & := \{u\colon X\to \R\cup\{\pm\infty\}:\; u\,\text{measurable},\, |u|<\infty \,\text{$m$-a.e.},\\
  & \qquad\qquad\qquad\qquad\qquad\qquad \exists (u_n) \text{ in } \calD: \calE[u_n-u_m]\to 0,\, u_n\to u\,\text{$m$-a.e.}\}.
\end{align*}
Clearly, $\calD_\e$ is a vector space containing $\calD$, and by \cite[Theorem 1.5.1]{Fukushima} we can extend $\calE$ to $\calD_\e$ by
\[\calE[u]:=\lim_{n\to\infty} \calE[u_n]\]
for $u\in\calD_\e$, where $(u_n)$ is a corresponding approximating sequence. By \cite[Theorem 2.1.7]{Fukushima}, every element in $\calD_\e$ admits a quasi-continuous representative,
so without loss of generality we may assume that the elements of $\calD_\e$ are quasi-continuous.
Note that $\calE$ is a positive quadratic form on $\calD_\e$,
but $(\calD_\e,\calE)$ may not be a Hilbert space.
However, if $\calE$ is a scalar product on $\calD$, then $(\calD_\e,\calE)$ is a Hilbert space (this is the so-called \emph{transient} case) and $\calD_\e$ can be identified with the abstract completion of $\calD$ w.r.t.\ $\calE$.
We can decompose $\calD_\e$ into an $\calE$-orthogonal sum
\[
  \calD_\e = \calD_{\e,X\setminus\tilde{F}}\oplus\mathcal{H}_{\tilde{F}}:= \{u\in\calD_\e:\; u=0\,\text{q.e.\ on } \tilde{F}\} \oplus\{P_{\tilde{F}} u:\; u\in\calD_\e\},
\]
where $\tilde{F}$ is a so-called quasi-support of $\mu$ and $P_{\tilde{F}}$ is given by a probabilistic expectation
\[P_{\tilde{F}} u = \E_{(\cdot)}\bigl(u(X_{\sigma_{\tilde{F}}})\bigr)\]
for $u\in \calD_\e$; cf.\ \cite[Section 6.2]{Fukushima} for details.
In case $\calE$ is a scalar product on $\calD$ we obtain that $P_{\tilde{F}}$ is an orthogonal projection on $\calD_\e$ w.r.t.\ $\calE$.

We define the form $q$ in $L^2(F,\mu)$ by
\begin{align*}
  \dom q & := \{\varphi\in L^2(F,\mu):\; \exists u\in \calD_\e: u=\varphi\,\text{$\mu$-a.e.}\},\\
  q[\varphi] & := \calE[P_{\tilde{F}}u],\quad\text{where $u\in\calD_\e$ with $u=\varphi$ $\mu$-a.e.}
\end{align*}
By \cite[Lemma 6.2.1]{Fukushima}, $q$ is well-defined. Note that for $\varphi\in \dom q$ we have $\varphi = P_{\tilde{F}} u$ $\mu$-a.e.
By \cite[Theorem 6.2.1]{Fukushima}, $q$ is a regular Dirichlet form, so in particular $q$ is closed. Moreover, $J(C_c(X)\cap\calD)\subseteq \ran J$ is a core for $q$.

\begin{prop}
  We have $\check{\calE} = q$.
\end{prop}

\begin{proof}
  By \cite[Theorems 4.6.2 and 4.6.5]{Fukushima} we observe the $\calE$-orthogonal decomposition
  \[\calD_\e = \{u\in\calD_\e:\; u=0\,\text{$\mu$-a.e.}\}\oplus \{P_{\tilde{F}}u:\; u\in\calD_\e\}.\]


  Let $\lambda>0$. Let $\varphi\in\ran J$, $u\in \dom J$ such that $Ju=\varphi$. Then
  \[\check{\calE}_\lambda[\varphi] = \calE_\lambda [P_\lambda u] = \calE[P_\lambda u] + \lambda \int (P_\lambda u)^2\,dm.\]
  Since $P_\lambda u = \E_{(\cdot)} \bigl(e^{-\lambda \sigma_{\tilde{F}}} u(X_{\sigma_{\tilde{F}}})\bigr)$ by \cite[Theorem 4.3.1]{Fukushima}, we obtain
  $P_\lambda u \to P_{\tilde{F}} u$ $m$-a.e.
  For $u\in \calD\cap C_c(X)$ we obtain
  \[\int |P_\lambda u|^2\, dm \leq \int P_{\tilde{F}} u^2 \, dm \leq \|u\|_\infty^2 m(\supp u).\]
  Since $P_\lambda u = u$ q.e.\ on $\tilde{F}$ by \cite[Theorem 4.3.1]{Fukushima}, we have $P_\lambda u = u = Ju$ $\mu$-a.e.
  Hence, $q[Ju] = \calE[P_{\tilde{F}} P_\lambda u] = \calE[P_\lambda u]$, since $P_{\tilde{F}}P_\lambda = P_\lambda$ by the tower property for conditional expectations.
  Therefore, for $\varphi\in J\bigl(\calD\cap C_c(X)\bigr)$ and $u\in \calD\cap C_c(X)$ such that $Ju=\varphi$ we obtain
  \[\check{\calE}_\lambda[\varphi] = \calE[P_\lambda u] + \lambda \int (P_\lambda u)^2\,dm = q[\varphi] + \lambda \int (P_\lambda u)^2\,dm \to q[\varphi].\]
  Hence, $\check{\calE}[\varphi] = \check{\calE}_0[\varphi] = q[\varphi]$.
  Since $J\bigl(\calD\cap C_c(X)\bigr)$ is a core for $q$ by \cite[Theorem 6.2.1]{Fukushima} and it is a core for $\check{\calE}$ by (the proof of) Proposition \ref{prop:checkcale_regular},
  we obtain $\check{\calE}= q$.
\end{proof}

\section{Convergence of traces of Dirichlet forms}
\label{sec:convergence_Dirichlet_forms}

Let $X$ be a locally compact separable metric space, $m$ a positive Radon measure on $X$ with full support $X$ and $\calE$ a regular Dirichlet form having domain $\calD\subseteq L^2(X,m)$.
Let $\mu$ be a positive Radon measure on $X$ charging no set of zero capacity.
We consider a sequence $(\calEn)$ of regular Dirichlet forms with $\dom \calEn = \calD$ for all $n\in\N$, and a Dirichlet form $\calE^\infty$ with domain $\dom \calE^\infty = \calD$.

We make the following three assumptions. First, assume there exists a constant $c>0$ such that
\begin{equation}
\label{eq:A1}
  c^{-1}\calE[u]\leq \calEn[u]\leq c\calE[u] \quad\text{for all } u\in\calD, n\in\N\cup\{\infty\}. \tag{A.1}
\end{equation}
Assumption \eqref{eq:A1} implies in particular that $\calE$ and $\calEn$ induce equivalent capacities.
Hence we shall use deliberately the abbreviations ``q.e.''\ and ``q.c.''\ to mean with respect to any of these capacities.
The second assumption that we will adopt is
\begin{equation}
\label{eq:A2}
  J_1\from (\calD\cap L^2(X,\mu),\calE_1)\to L^2(X,\mu),\quad u\mapsto u \quad \text{is continuous}. \tag{A.2}
\end{equation}
Note that since $J_1$ is densely defined by Lemma \ref{lem:properties_of_J}, we can then extend $J_1$ to $\calD$.

For $n\in\N\cup\{\infty\}$ we define as before
\[
  \Jn \from (\calD\cap L^2(X,\mu),\calEn_1) \to L^2(X,\mu),\quad  u\mapsto u.
\]
By \eqref{eq:A1} and \eqref{eq:A2} also $\Jn$ is continuous and can be extended to $\calD$.
For the third assumption, for $n\in\N\cup\{\infty\}$ let $H^n$ be the positive self-adjoint operator associated with $\calE^n$ and $K^n:=(H^n+1)^{-1}$.
Then we assume that for all $u\in L^2(X,m)$ we have
\begin{equation}
\label{eq:A3}
  J_1^{n} K^{n} u \to J_1^\infty K^\infty u\quad \text{in } L^2(X,\mu).	\tag{A.3}
\end{equation}
For $n\in\N\cup\{\infty\}$ and $\lambda>0$ we denote by $\check{\calE}^n_\lambda$ the trace of the Dirichlet form $\calEn_\lambda$ w.r.t.\ the measure $\mu$.

Let us recall the definition of Mosco convergence, see \cite[Definition 2.1.1]{Mosco1994} or \cite{Mosco1969}.
Let $(\form_n)$ be a sequence of positive quadratic forms in a Hilbert space $\calH$, $\form_\infty$ a quadratic form in $\calH$.
We say that $(\form_n)$ \emph{Mosco-converges} to the form $\form_\infty$ in $\calH$ provided
\begin{enumerate}
  \item[(M1)]
    for all $(u_n)$ in $\calH$, $u\in \calH$ such that $u_n\to u$ weakly in $\calH$ we have $\liminf_{n\to\infty} \form_n[u_n] \geq \form_\infty[u]$,
  \item[(M2)]
    for all $u\in \calH$ there exists $(u_n)$ in $\calH$ such that $u_n\to u$ in $\calH$ and $\limsup_{n\to\infty} \form_n[u_n]\leq \form_\infty[u]$.
\end{enumerate}
Note that for this definition we extend the quadratic forms to the whole space by setting them $+\infty$ for elements not in their domain.

\begin{theo}
  \label{convtrace}
  Assume \eqref{eq:A1}, \eqref{eq:A2} and \eqref{eq:A3}.
  Let $(\calEn)$ be Mosco-convergent to $\calE^\infty$.
  Then:
  
  {\rm(a)}
      The sequence of trace forms $(\check{\calE}^n_\lambda)$ Mosco-converges to the corresponding trace form $\check{\calE}_\lambda^\infty$ for every $\lambda>0$.
      
  {\rm(b)}
      For every sequence $(\lambda_j)$ in $(0,\infty)$ such that $\lambda_j\downarrow 0$ there exists a sequence $(n_j)$ in $\N$ with $n_j\to\infty$ such that $(\check{\calE}^{n_j}_{\lambda_j})$ Mosco-converges to the trace form $\check{\calE}^\infty$.
\end{theo}

\begin{proof}
  First, note that $J^\infty_1 u = \Jn u = J_1 u$ for all $u\in\calD$ and $n\in\N$.

  (a)
  We shall prove the statement for $\lambda=1$, the proof for general $\lambda>0$ is similar.
  For $n\in\N\cup\{\infty\}$ we define the bounded form $Q_n$ by
  \[
    \dom Q_n := L^2(X,\mu),\quad Q_n[\psi] := \calEn_1[(\Jn)^*\psi].
  \]
  Since $(\Jn)^*\psi \in \bigl(\ker (\Jn)\bigr)^{\perp_{\calE^n_1}}$, from the very definition we obtain
  \[
    Q_n[\psi] = \check{\calE}^n_1[\Jn (\Jn)^*\psi] = \int_X \psi\cdot \Jn(\Jn)^*\psi\,d\mu.
  \]
  Hence $Q_n$ is the closed quadratic form associated to the positive self-adjoint bounded operator $\Jn(\Jn)^*=(\check{H}^n_1)^{-1}$, where $\check{H}^n_1$ is the operator associated with $\check{\calE}^n_1$.
  As Mosco-convergence for forms is equivalent to strong resolvent convergence for the associated operators (see \cite[Theorem 2.4.1]{Mosco1994}) and for bounded self-adjoint operators strong convergence and resolvent convergence are equivalent
  we are led to prove that $(Q_n)$ Mosco-converges to $Q_\infty$.

  To prove (M1), let $(\psi_n)$ be a $L^2(X,\mu)$-weakly convergent sequence with  weak limit $\psi\in L^2(X,\mu)$. W.l.o.g.\ we may assume that $\liminf_{n\to\infty} Q_n[\psi_n] = \lim_{n\to\infty} Q_n[\psi_n]$ (otherwise choose a suitable subsequence).
  First, note that $(\psi_n)$ is bounded. By \eqref{eq:A1} and \eqref{eq:A2} we easily obtain that $\sup_{n\in\N\cup\{\infty\}}\|(\Jn)^*\| = \sup_{n\in\N\cup\{\infty\}}\|\Jn\|<\infty$. Thus,
  \[\sup_{n\in\N} Q_n[\psi_n] = \sup_{n\in\N} \calEn_1[(\Jn)^*\psi_n] \leq \sup_{n\in\N} \|(\Jn)^*\|^2 \|\psi_n\|^2 < \infty.\]
  In particular,
  \begin{align}
    \sup_{n\in\N} \calEn_1[(\Jn)^*\psi_n]<\infty.
  \label{En-bounded}
  \end{align}
  For the rest of the proof we shall use Kuwae's method (see \cite[Ssection 2.2]{Kuwae}) as follows:
  For $n,m\in\N\cup\{\infty\}$ define
  \[
  E_n:=({\mathcal{D}},(\calEn_1)^{1/2}),\quad  {\mathcal{C}}_n:=\ran \Kn,\quad
  \Phi_{m,n}\from {\mathcal{C}}_m \to E_n,\quad
    \Phi_{m,n}u:=\Kn(K^{m})^{-1}u.
  \]
  Then $\Phi_{m,m}u = u$ for all $u\in \mathcal{C}_m$.
  Furthermore, as Mosco-convergence of forms is equivalent to strong resolvent convergence of the associated operators, for $v\in \mathcal{C}_\infty = \ran K^\infty$ and $u\in L^2(X,m)$ such that $K^\infty u = v$ we get
  \begin{equation}
  \label{eq:calEn_[Kn_u]_convergence}
    \calEn_1[\Phi_{\infty,n}v] = \calEn_1[\Kn u]=\int_X u \Kn u\,dm\to \int_X u K^\infty u\,dm = \calE_1^{\infty}[K^\infty u] = \calE_1^\infty[v].
  \end{equation}
  Hence, $(E_n)_{n\in\N}$ converges to $E_\infty$ in the sense of Kuwae and assumption \cite[Assumption A.2.1]{Kuwae} is fulfilled.

  Let $u\in L^2(X,m)$. By \eqref{eq:A3} we have $\Jn\Kn u \to J_1^\infty K^\infty u$ in $L^2(X,\mu)$.
  For $n\in\N$ define $w_n:= K^\infty u\in \ran K^\infty$. Then clearly $w_n\to K^\infty u$ in $E_\infty$, and
  \[\lim_{k\to\infty} \limsup_{n\to\infty}  \calEn_1[\Phi_{\infty,n} w_k - \Kn u] = \lim_{k\to\infty} \limsup_{n\to\infty} \calEn_1[\Kn u -\Kn u] = 0.\]
  Hence, $\Kn u \to K^\infty u$ strongly in the sense of Kuwae (see \cite[Definition 2.4]{Kuwae}).


  By \eqref{En-bounded}, an application of \cite[Lemma 2.2]{Kuwae} yields that there exists a subsequence $\bigl((J^{n_k}_1)^*\psi_{n_k}\bigr)_k$ and $u_\infty\in E_\infty$ such that for all $u\in L^2(X,m)$ we have
  \[\lim_{k\to\infty} \calE^{n_k}_1\bigl((J^{n_k}_1)^*\psi_{n_k}, K^{n_k} u\bigr) = \calE_1^\infty(u_\infty,K^{\infty} u)\]
  (weak convergence in the sense of Kuwae \cite[Definition 2.5]{Kuwae}).
  Since $(\psi_n)$ is weakly convergent to $\psi$ and $(\Jn \Kn u)$ is strongly convergent to $J_1^\infty K^\infty u$ we also obtain
  \[\calE^{n_k}_1\bigl((J^{n_k}_1)^*\psi_{n_k}, K^{n_k} u\bigr) = \int \psi_{n_k} J^{n_k}_1 K^{n_k} u \,d\mu \to \int \psi J^\infty_1 K^\infty u\,d\mu = \calE^\infty_1\bigl((J^\infty_1)^* \psi, K^\infty u\bigr).\]
  Thus, $\calE_1^\infty(u_\infty,K^{\infty} u) = \calE^\infty_1\bigl((J^\infty_1)^* \psi, K^\infty u\bigr)$ for all $u\in L^2(X,m)$. Since $\ran K^\infty = \dom H^\infty$ is a core for $\calE^\infty_1$, we conclude
  $u_\infty = (J^\infty_1)^*\psi$.

  From \cite[Lemma 2.3]{Kuwae} we then get
  \begin{align*}
    \liminf_{n\to\infty}Q_{n}[\psi_{n}] & = \lim_{n\to\infty}Q_{n}[\psi_{n}] = \liminf_{k\to\infty}Q_{n_k}[\psi_{n_k}] = \liminf_{k\to\infty}\calE^{n_k}_1[(J^{n_k}_1)^*\psi_{n_k}] \\
    & \geq \calE_1^\infty[(J^\infty_1)^*\psi] = Q_\infty[\psi],
  \end{align*}
  and (M1) is proved.

  To prove (M2), let $\psi\in L^2(X,\mu)$. We will use $\psi_n:=\psi$ for all $n\in\N$.
  Without loss of generality, we may assume that $\limsup_{n\to\infty} Q_n[\psi] = \lim_{n\to\infty} Q_n[\psi]$ (otherwise choose a suitable subsequence).
  By \eqref{En-bounded} and \eqref{eq:A1} the sequence $\bigl((J^n_1)^*\psi\bigr)_n$ is bounded with respect to $\calE^\infty_1$.
  By choosing a suitable subsequence, we may assume that $\bigl((J^n_1)^*\psi\bigr)_n$ converges weakly to some $u_\infty\in\calD$ with respect to $\calE^\infty_1$.
  Thus,
  \[\int_X u (J^{n}_1)^*\psi\,dm = \calE^\infty_1\bigl(K^\infty u,(J^{n}_1)^*\psi\bigr) \to \calE^\infty_1(K^\infty u,u_\infty) = \int_X u u_\infty\,dm\]
  for all $u\in L^2(X,m)$.
  By \eqref{En-bounded} and reasoning as in the proof of (M1) the sequence $\bigl((J^n_1)^*\psi\bigr)_n$ has a subsequence $\bigl((J^{n_k}_1)^*\psi\bigr)_k$ which converges weakly in the sense of Kuwae to some $u_\infty'\in E_\infty$.
  In particular,
  \[\int_X u (J^{n_k}_1)^* \psi\,dm = \calE^{n_k}_1\bigl(K^{n_k} u, (J^{n_k}_1)^* \psi\bigr) \to \calE^{\infty}_1(K^\infty u, u_\infty') = \int_X u u_\infty'\,dm\]
  for all $u\in L^2(X,m)$. Thus, $u_\infty=u_\infty'$.
  Since also
  \[\int_X u (J^{n_k}_1)^* \psi\,dm = \calE^{n_k}_1\bigl(K^{n_k} u, (J^{n_k}_1)^* \psi\bigr) \to \calE^{\infty}_1\bigl(K^{\infty} u, (J^{\infty}_1)^* \psi\bigr) = \int_X u (J^{\infty}_1)^* \psi\,dm\]
  for all $u\in L^2(X,m)$ as in the proof of (M1), we obtain $u_\infty = (J^{\infty}_1)^* \psi$.
  Since $J^\infty_1$ is linear and continuous, it is also weakly continuous. Hence, $J^{n_k}_1 (J^{n_k}_1)^*\psi = J^\infty_1 (J^{n_k}_1)^*\psi \to J^\infty_1 (J^\infty_1)^*\psi$ weakly in $L^2(X,\mu)$.
  Thus,
  \begin{align*}
    \limsup_{n\to\infty} Q_n[\psi] & = \lim_{n\to\infty} Q_n[\psi] = \lim_{k\to\infty} Q_{n_k}[\psi] = \lim_{k\to\infty} \calE^{n_k}_1[(J^{n_k}_1)^*\psi] \\
    & = \lim_{k\to\infty} \int \psi J^{n_k}_1 (J^{n_k}_1)^*\psi \,d\mu = \int \psi J^{\infty}_1 (J^{\infty}_1)^*\psi \,d\mu = \calE^\infty_1[(J^\infty_1)^*\psi] = Q_\infty[\psi].
  \end{align*}

  (b)
  According to \cite[Theorem 3.36]{Attouch},
  the topology of Mosco-convergence on the space of closed forms on a Hilbert space is metrizable.
  Thus (b) is simply a consequence of a diagonal procedure.
\end{proof}

\begin{rk}
\label{FinalRemark}
  {\rm(a)}
      For \eqref{eq:A1} it suffices to require that there exists $c>0$ such that
      \[c^{-1}\calE[u]\leq \calEn[u]\leq c\calE[u] \quad\text{for all } u\in\calD, n\in\N.\]
      Then Mosco-convergence of $(\calEn)$ to a Dirichlet form $\calE^\infty$ yields \eqref{eq:A1}.
      
  {\rm(b)}
      Note that in Theorem \ref{convtrace}, compared to \cite[Theorem 3.7]{Arendt1}, we just require Mosco convergence of $(\calE^n)$ to $\calE^\infty$.
\end{rk}

The following lemma can be used to obtain \eqref{eq:A3}.

\begin{lem}
  Assume \eqref{eq:A1} and \eqref{eq:A2}. Let $(\calEn)$ be Mosco-convergent to $\calE^\infty$ and assume that
  $\calEn[v]\to \calE^\infty[v]$ for all $v\in \ran K^\infty$. Then \eqref{eq:A3} is satisfied.
\end{lem}

\begin{proof}
  We make use of the notation introduced in the proof of Theorem \ref{convtrace}.
  Let $u\in L^2(X,m)$. 
  Since $\Kn u \to K^\infty u$ in $L^2(X,m)$ (cf.\ Remark \ref{rem:construction}(b)), as in \eqref{eq:calEn_[Kn_u]_convergence} we obtain
  \begin{align*}
    \calEn_1[\Kn u - K^\infty u] & = \calEn_1[\Kn u] - 2\calEn_1(\Kn u, K^\infty u) + \calEn_1[K^\infty u]\\
    & = \calEn_1[\Kn u] - 2 \int_X u K^\infty u\,dm + \calEn_1[K^\infty u]\\
    & \to \calE^\infty_1[K^\infty u] - 2 \int_X u K^\infty u\,dm + \calE^\infty_1[K^\infty u]\\
    & = \calE^\infty_1[K^\infty u] - 2\calE^\infty_1[K^\infty u] + \calE^\infty_1[K^\infty u] = 0.
  \end{align*}
  By \eqref{eq:A1} we conclude
  \[\calE^\infty_1[\Kn u - K^\infty u] \to 0.\]
  Since $J^\infty_1$ is continuous, we have $J_1^nK^n u = J_1^\infty K^n u \to J_1^\infty K^\infty u$ in $L^2(X,\mu)$.
\end{proof}

The following counter-example shows that if \eqref{eq:A1} fails (whereas \eqref{eq:A2} still hold true) then the conclusions of Theorem \ref{convtrace} may fail!

\begin{exa}
  Let $X:=[0,1]$, $m$ the Lebesgue measure on $X$, $\calH:=L^2(X,m) = L^2(0,1)$, $\calE$ the classical Dirichlet form with Neumann boundary conditions, i.e.
  \[
    \calD := H^1(0,1),\quad \calE[u] := \int_0^1 u'(x)^2\,dx,
  \]
  and $\mu := \delta_0+\delta_1$. For $n\in\N$ define $\calEn$ in $\calH$ by
  \[
    \calEn[u] := \frac{1}{n}\int_0^1 u'(x)^2\,dx + u(0)^2 + u(1)^2 \quad\text{for all } u\in\calD.
  \]
  We shall identify the space $L^2(X,\mu)$ with the Euclidean space $\R^2$.
  Clearly, $\calE$ is a regular Dirichlet form. By Sobolev embedding, elements from $\calD$ have continuous representatives. Moreover  we see that $J$ is densely defined with dense range, $J_1$ is bounded and the $\calEn$'s are closed. This is indeed all we need.
  Furthermore, for every $n\in\N$ and $u\in\calD$ we have
  \[
    \calEn[(0\vee u)\wedge 1]\leq \calEn[u].
  \]
  Hence the $\calEn$'s are Dirichlet forms. However, assumption \eqref{eq:A1} is not fulfilled in this particular case. Indeed, for $u\in H^1_0(0,1)\subseteq \calD$ and $n\in\N$ we observe $\calEn[u] = \frac{1}{n}\calE[u]$.

  By \cite[Theorem 3.2]{Simon} the sequence $(\calEn)$ converges in the sense of Mosco to the closure of the regular part of the quadratic form $q$ defined by
  \[
    \dom q:=\calD,\quad  q[u] := u(0)^2 + u(1)^2.
  \]
  However, it is well-known that $q = q_{\rm sing}$ (compare \cite[Example (3)]{Simon}; but it is also easy to see) and hence $\calEn\to 0$ in the sense of Mosco.

  Obviously, in this situation we have $\check{\calE}^\infty =0$.
  Furthermore, for every $\lambda>0$, the $\calE^\infty_\lambda$-orthogonal complement of $\ker J$ is $\{0\}$.
  Hence, $\check{\calE}_\lambda^\infty=0$ for all $\lambda\geq 0$.

  We shall show that $\check{\calE}^n_\lambda\not\to 0$ in the sense of Mosco, for any $\lambda\geq 0$.
  Let us first compute $\check{\calE}^n_\lambda$ for $\lambda\geq 0$. To this end, for given $a,b\in\R$, we solve the boundary value problem
  \begin{align*}
    -\frac{1}{n}u'' + \lambda u & =0 \quad \text{in } (0,1),\\
    u(0) & =a,\quad u(1) = b.
  \end{align*}
  For $\lambda>0$, the solution is given by
  \[
    u_{a,b}(x) = au_1(x) + bu_2(x)  \quad\text{for all } x\in [0,1],
  \]
  where
  \[
  u_1(x) := \frac{\sinh (\sqrt{n\lambda}(1-x))}{\sinh\sqrt {n\lambda}}, \quad
  u_2(x) := \frac{\sinh (\sqrt{n\lambda}x)}{\sinh\sqrt {n\lambda}}  \quad\text{for all } x\in[0,1].
  \]
  From the definition of $\check{\calE}^n_\lambda$ we get
  \begin{align*}
    \check{\calE}^n_\lambda[(a,b)] & = \calEn_\lambda[u_{a,b}] = \frac{1}{n}u_{a,b}(1)u_{a,b}'(1)-\frac{1}{n}u_{a,b}(0)u_{a,b}'(0) + a^2 + b^2 \\
    & = -\frac{\sqrt{n\lambda}}{\sinh(\sqrt{n\lambda})}\frac{2ab}{n} + \frac{(a^2+b^2)\sqrt{n\lambda}}{n}\frac{\cosh(\sqrt{n\lambda})}{\sinh(\sqrt{n\lambda})} + a^2 + b^2.
  \end{align*}
  Moreover, for $\lambda=0$ an elementary computation yields
  \begin{align*}
    \check{\calE}^n[(a,b)] = \frac{1}{n} (b-a)^2 + a^2 + b^2 \quad\text{for all } a,b\in\R.
  \end{align*}
  Therefore, for all $\lambda\geq 0$ we obtain
  \[
    \lim_{n\to\infty}\check{\calE}^n_\lambda[(a,b)]=a^2+b^2.
  \]
  Since the limit form is bounded on $\R^2$, by \cite[Theorem 3.2]{Simon} we conclude that for each $\lambda\geq 0$ the sequence $(\check{\calE}^n_\lambda)$ converges in the sense of Mosco to the Euclidean scalar product on $\R^2$.
\end{exa}

\section*{Acknowledgement}

We thank the referee for many valuable comments which improved the manuscript. In particular, they led to clarifications yielding Theorem \ref{thm:constructions_agree} and a strengthening of Lemma \ref{lem:calH_har_closed}.

\bibliography{BiblioConv}

\end{document}